\documentclass{amsart}

\usepackage{latexsym,amssymb}
\usepackage{amsmath,amsthm}

\usepackage{amsfonts}
\usepackage{graphicx}

\def\glim{\mathop{\text{\normalfont $\Gamma-$lim}}}
\def\supp{\mathop{\text{\normalfont supp}}}
\def\dist{\mathop{\text{\normalfont dist}}}

\newcommand{\R}{\mathbb{R}}

\newcommand{\N}{\mathbb{N}}

\newcommand{\ve}{\varepsilon}
\newcommand{\ito}{\infty}
\newcommand{\T}{\mathcal{T}}

\newtheorem{thm}{Theorem}[section]

\newtheorem{lemp}[thm]{Lemma}
\newtheorem{prop}[thm]{Proposition}
	
\theoremstyle{definition}
\newtheorem{de}[thm]{Definition}
\theoremstyle{remark}
\newtheorem{rk}[thm]{Remark}

\numberwithin{equation}{section}

\begin{document}

\title[Shape optimization]{A shape optimization problem for Steklov eigenvalues in oscillating domains}

\author[J. Fern\'andez Bonder and J. F. Spedaletti]{Juli\'an Fern\'andez Bonder and Juan F. Spedaletti}

\address[J. Fern\'andez Bonder]{Departamento de Matem\'atica FCEN - Universidad de Buenos Aires and IMAS - CONICET. Ciudad Universitaria, Pabell\'on I (1428)
Av. Cantilo s/n. Buenos Aires, Argentina.}

\address[J. F. Spedaletti]{Departamento de Matem\'atica, Universidad Nacional de San Luis and IMASL - CONICET.}

\email[J. Fernandez Bonder]{jfbonder@dm.uba.ar}
\urladdr[J. Fernandez Bonder]{http://mate.dm.uba.ar/~jfbonder}

\email[J. F. Spedaletti]{jfspedaletti@unsl.edu.ar}



\begin{abstract}
In this paper we study the asymptotic behavior of some optimal design problems related to nonlinear Steklov eigenvalues, under irregular (but diffeomorphic) perturbations of the domain.
\end{abstract}

\maketitle

\section{Introduction}

Let $\Omega\subset \mathbb R^n$ be a bounded domain with regular boundary, let $\alpha\in (0,1)$ and $\Gamma\subset \partial\Omega$ be a measurable set (a window) such that $|\Gamma|_{n-1} = \alpha |\Omega|_{n-1}$, where $|\cdot|_d$ refers to the $d-$dimensional Hausdorff measure. The optimal Sobolev trace constant is defined as
$$
\lambda(\Gamma) := \inf_{v\in W^{1,p}_\Gamma(\Omega)} \frac{\int_\Omega |\nabla v|^p + |v|^p\, dx}{\int_{\partial\Omega} |v|^p\, dS},
$$
where $W^{1,p}_\Gamma(\Omega)$ is the set of functions $v\in W^{1,p}(\Omega)$ such that $v|_{\Gamma}=0$.

In \cite{Bonder-Neves-Delpezzo}, the authors study the following problem: minimize $\lambda(\Gamma)$ among all admissible windows, i.e.
\begin{equation}\label{lambda.alpha}
\lambda(\alpha) = \inf_{\Gamma\in \Sigma_\alpha} \lambda(\Gamma), 
\end{equation}
where $\displaystyle \Sigma_{\alpha} = \{\Gamma\subset \partial\Omega\colon \text{are measurable and } \displaystyle |\Gamma|_{n-1} = \alpha |\partial\Omega|_{n-1}\}.$

In the above mentioned work the authors show the existence of an {\em optimal window} $\Gamma_0$, i.e. some $\Gamma_0\in \Sigma_\alpha$ such that $\lambda(\alpha)=\lambda(\Gamma_0)$. Moreover it is shown that if $u_0$ is the eigenfunction associated to $\lambda(\Gamma_0)$ then $\{u_0=0\}\cap\partial\Omega = \Gamma_0$.

We refer the interested reader to \cite{Bonder-Neves-Delpezzo} and references therein for a motivation and history of this problem.

In this work we study the behavior of this optimal windows when the domain $\Omega$ is perturbed periodically by a sequence of domains $\Omega_\ve$ and try to determine whether they approximate $\Gamma_0$ in some reasonable sense.

Let us denote by $\lambda_\ve(\alpha)$ the constant \eqref{lambda.alpha} in the domain $\Omega_\ve$. We find that the behavior of the constants $\lambda_\ve(\alpha)$ and of their corresponding optimal windows $\Gamma_\ve$ depend strongly on the amplitude of the oscillations. We distinguish three cases: i.- Subcritical case: In this case the oscillations are very big and the trace constant converges to zero. ii.- Supercritical case: In this case the oscillations are very small and there are convergence to the unperturbed problem. iii.- Critical case: In this case the amplitude compensates with the oscillations and this is reflected in the appearance of a weight term. 

The results presented here are new even in the linear eigenvalue problem that corresponds to $p=2$.

\subsection{$\ve$-Oscilations}

In \cite{Bonder-Neves-Delpezzo}, the authors studied the asymptotic behavior of $\lambda_\ve(\alpha)$ where the domains $\Omega_\ve$ are {\em regular} perturbations of the original domain $\Omega$. To be precise, the authors apply the so-called Hadamard variations of domains method and are able to compute the {\em shape derivative} of $\lambda(\alpha)$ with respect to these deformations. See \cite{Bonder-Neves-Delpezzo} for the details.

Here we follow a different path. Instead of considering regular perturbations we analyze the case of periodic oscillatory deformations where the amplitude of these oscillations converge to zero, and the period of these oscillations also converge to zero.

We start by describing the type of perturbations that we are to consider. Let $\Omega\subset \R^{n}$ be bounded. Assume that the boundary is regular ($C^1$ will be enough for most of our arguments). Take $U\subset \R^n$ and $\Phi\colon U'\subset \R^{n-1}\to \R$, where $U'$ is open and connected, such that
\begin{align*}
\partial\Omega\cap U &= \{(x_1,x')\in \R^n \colon  x'\in U',\  x_1=\Phi(x')\},\\
\Omega\cap U &= \{(x_1,x')\in \R^n \colon  x'\in U',\  x_1<\Phi(x')\}.
\end{align*}

Let $f\colon \R^{n-1}\to \R$ be a $C^1$ function, periodic with period $Y'=[0,1]^{n-1}$. 

With all this notation we can now define our perturbed domains $\Omega_\ve\subset \R^n$ as
\begin{equation}\label{Omega.epsilon}
\Omega_{\ve}\cap U=\{(x_{1},x')\in U \colon x'\in U',\ x_{1}<\Phi(x')+\ve^{a}f(\tfrac{x'}{\ve})\}
\end{equation}
and therefore, 
$$
\partial\Omega_{\ve}\cap U = \{(x_{1},x')\in \R^{N} \colon x'\in U',\ x_{1}=\Phi(x')+\ve^{a}f(\tfrac{x'}{\ve}) \}.
$$

By the results of \cite{Bonder-Neves-Delpezzo}, for every constant $\lambda_\ve(\alpha)$, there exists an optimal window $\Gamma_\ve$ and the corresponding eigenfunction $u_\ve\in W^{1,p}(\Omega_\ve)$ verifies that $\Gamma_\ve = \{u_\ve=0\}\cap \partial\Omega_\ve$. Our goal is to study the behavior of these optimal windows $\Gamma_\ve$, their eigenfunctions $u_\ve$ and of the constants $\lambda(\Gamma_\ve)= \lambda_\ve(\alpha)$ when $\ve\downarrow 0$.

Observe that these domains $\Omega_\ve$ converge to $\Omega$ in practically any reasonable notion of set convergence in $\R^n$ (for instance in the Hausdorff complementary topology, the $L^1$ norm of the characteristic functions, etc.).

As we mentioned in the introduction, the behavior strongly depends on the amplitude of the oscillations measured in terms of the parameter $a>0$.

Three cases appear: 
\begin{itemize}
\item The {\em subcritical} case, that corresponds to large oscillations with respect to the period ($a<1$).

\item The {\em supercritical} case, that corresponds to small oscillations with respect to the period ($a>1$).

\item The {\em critical} case, that corresponds to the case where amplitude and oscillations are of the same order ($a=1$).
\end{itemize}

In the subcritical case, being the oscillations so big, the problem degenerates and the immersion is lost in the limit. This is a {\em fattening} phenomena of the boundary and it is reflected in the fact that the constants $\lambda(\Gamma_\ve)$ converge to zero.

In the supercritical case, the oscillations are too small. Then, for small values of $\ve$ the oscillations become imperceptible and that is reflected in the fact that the problem converges to the unperturbed one when $\ve\downarrow 0$.

Finally, the critical case is the most interesting. In this case, the oscillations and the periods are balanced and an {\em homogeneization} phenomena appears at the boundary. This homogenization is reflected in the appearance of a {\em strange term} at the boundary for the limit problem in the spirit of Cioranescu-Murat \cite{Murat}. This phenomena have been observed in the work \cite{Bonder-Orive-Rossi} where the pure eigenvalue problem is addressed.

Taking into account the above perturbation of the domain $\Omega$ we get the result.

\begin{thm}\label{teo.main}
Let $\Omega\subset \R^n$ be an open, bounded set and assume that $\partial\Omega$ is of class $C^1$. Let $\{\Omega_\ve\}_{\ve>0}$ be the family of perturbed domains as described in \eqref{Omega.epsilon}. Let $\lambda_{\ve}(\alpha)$ ($0<\alpha<1$) be the best Sobolev trace constant on $\Omega_{\ve}$ given by \eqref{lambda.alpha} in the domain $\Omega_\ve$. 

Then the following statements hold true:
\begin{enumerate}
\item {\em (Subcritical case)} If $a<1$ then $\lim_{\ve\to 0}\lambda_{\ve}(\alpha)=0$, moreover, we have the following asymptotic behavior
\begin{equation}\label{subcritico}
\lambda_{\ve}(\alpha)\leq C \ve^{1-a}
\end{equation}
where the constant $C$ depends only on the function $f$ used in the perturbation.

\item {\em (Supercritical case)} If $a>1$ then $\lim_{\ve\to 0}\lambda_{\ve}(\alpha)= \lambda(\alpha)$.

\item {\em (Critical case)} If $a=1$ then $\lim_{\ve\to 0}\lambda_{\ve}(\alpha) = \lambda^{*}(\alpha)$, where $\lambda^{*}(\alpha)$ is defined as
\begin{equation}\label{lambdaestrella}
\lambda^{*}(\alpha) := \inf \left \{\frac{\int_{\Omega}|\nabla u|^{p}+|u|^{p}dx}{\int_{\partial \Omega} |u|^{p}\, d\mu^*} \colon u\in W^{1,p}(\Omega),\ \mu^*(\{u=0\}\cap \partial\Omega) \ge \alpha \mu^*(\partial\Omega)\right\},
\end{equation}
where the measure $\mu^*$ is given by $d\mu^* = m dS$ and the weight $m$ is defined by
\begin{equation}\label{m}
m(x)=\frac{\int_{Y}\sqrt{1+|\nabla \Phi(x')+\nabla f(y)|^{2}}dy}{\sqrt{1+|\nabla \Phi(x')|^{2}}}.
\end{equation}
\end{enumerate}
\end{thm}

Nevertheless, our method is far more general and we are able to treat general perturbations where the periodic perturbation described above is just an (important) example. See Theorem \ref{superycriticogeneral} below. In particular, the perturbations considered here also cover the regular deformations considered in \cite{Bonder-Neves-Delpezzo}.

Moreover, we go further and analyze the behavior of these optimal windows $\Gamma_\ve$ and of their corresponding eigenfunctions $u_\ve$ as $\ve\downarrow 0$. We found that, in the critical and in the subcritical case (an also in the more general framework of Theorem \ref{superycriticogeneral}) these optimal windows converge (in a suitable sense) to an optimal window of the corresponding limit problem and also the convergence of their eigenfunctions to the eigenfunction of the limit problem. See Theorem \ref{teo.ventanas}.

\subsection{Organization of the paper}

After this introduction, the paper is organized as follows. In Section \ref{changes.variables.sec}, we study the qualitative properties of the change of variables that deforms the original domain $\Omega$ into the periodically perturbed one $\Omega_\ve$. In Section \ref{subcritical.sec} we analyze the subcritical perturbation ($a<1$) in Theorem \ref{teo.main}. In Section \ref{critical.sec} we prove one of the main theorems of the paper (Theorem \ref{superycriticogeneral}) that implies, for instance, the critical ($a=1$) and the supercritical ($a>1$) cases in Theorem \ref{teo.main} and, moreover, the convergence of the corresponding eigenfunctions to the eigenfunction of the limit problem. Finally, in Section \ref{window.sec} we prove our second main theorem (Theorem \ref{teo.ventanas}) on the convergence of optimal windows.

\section{Estimates for the changes of variables}\label{changes.variables.sec}

In the analysis of the asymptotic behavior of the problem when $\ve\downarrow 0$, it is of fundamental importance to understand the asymptotic behavior of the changes of variables that take the perturbed domains $\Omega_\ve$ into $\Omega$.

Once these asymptotic behaviors are studied, the analysis is independent of the particular form of the change of variables and only depends on this asymptotic behavior.

Hence, given $\ve>0$ we define the transformation $T_\ve\colon \Omega_\ve\to \Omega$ as
\begin{equation}\label{cambiovariablesupercritico}
(y_1, y') = T_\ve(x_1, x') = (x_1 - \ve^{a}f(\tfrac{x'}{\ve})\phi_\ve(x), x'),
\end{equation}
where, as usual, $x'=(x_2,\dots,x_n)$ and $\phi_\ve\in C^\ito_c(\R^n)$ is supported on $B_{\sqrt{\ve}}(\partial\Omega) = \bigcup_{x\in \partial\Omega} B_{\sqrt{\ve}}(x)$, $\phi_\ve\equiv 1$ in $\partial\Omega$, $0\le\phi_\ve\le 1$, $|\nabla \phi_\ve|\le C\ve^{-\frac12}$.

We now compute the differential of $T_\ve$, $DT_\ve$.
$$
DT_\ve = \begin{pmatrix}
1-\ve^a f \partial_1 \phi_\ve & - \ve^{a-1}\partial_2 f \phi_\ve - \ve^a f \partial_2 \phi_\ve& \cdots & -\ve^{a-1} \partial_n f\phi_\ve - \ve^a f \partial_n\phi_\ve\\
0&  &  & \\
\vdots &  & I_{n-1 \times n-1} &  \\
0&  &  
\end{pmatrix}.
$$
Observe that
$$
DT_\ve(x) = I_{n\times n} -\ve^a f(\tfrac{x'}{\ve}) A_\ve(x) - \ve^{a-1} \phi_\ve(x) B(\tfrac{x'}{\ve}),
$$
where
$$
A_\ve(x) := \begin{pmatrix}
\nabla \phi_\ve(x)\\
0\\
\vdots\\
0
\end{pmatrix}, \qquad 
B(x') := \begin{pmatrix}
0 & \nabla f(x')\\
0 & 0\\
\vdots & \vdots\\
0 & 0
\end{pmatrix}.
$$

Finally, since $\|f\|_\infty<\infty$ and $\|\nabla f\|_\ito<\ito$ we have that
$$
\|B\|_\ito <\ito.
$$

Moreover, since $\|\nabla \phi_\ve\|_\ito\le C \ve^{-\frac12}$, we get
$$
\|A_\ve\|_\ito\leq C \ve^{-\frac12}\chi_{_{\supp(\phi_\ve)}}
$$  
and therefore we obtain that, calling $f_\ve(x') = f(\tfrac{x'}{\ve})$, 
$$
\|\ve^a f_\ve A_\ve\|_\ito\leq C \ve^{a-\frac12}\chi_{_{\supp(\phi_\ve)}}.
$$
On the other hand, calling $B_\ve(x') = B(\tfrac{x'}{\ve})$, 
$$
\|\ve^{a-1} \phi_\ve B_\ve\|_\ito \leq C \ve ^{a-1}\chi_{_{\supp(\phi_\ve)}}.
$$

Observe that when $a\ge 1$, we have that given $K\subset \Omega$ compact, $T_\ve = id_{\R^n}$ on $K$ for $\ve>0$ small enough. In particular
$$
DT_\ve = I_{n\times n} \quad \text{and} \quad JT_\ve = 1
$$
on $K$ for $\ve>0$ small, where $J T_\ve = |\det(DT_\ve)|$ is the Jacobian of $T_\ve$.

Finally, in the case $a>1$, $T_\ve\to id_{\R^n}$ in $C^1$ norm and, as a consequence, we get
$$
DT_\ve\rightrightarrows I_{n\times n},\quad JT_\ve\rightrightarrows 1\quad \text{and}\quad J_\tau T_\ve\rightrightarrows 1,
$$
where $J_\tau T_\ve=|DT_\ve^{-1} \mathbf{n}| JT_\ve$ is the tangential Jacobian of $T_\ve$ and $\mathbf{n}$ is the outer unit normal vector of $\Omega$. See \cite{Henrot} for more details on the tangential Jacobian.

We need now to study the asymptotic behavior of the tangential Jacobian in the case $a=1$. In this case, for $x\in \partial \Omega$ taking into account that $\phi_\ve =1$ on $\partial \Omega$ we get the following expression for the differential
$$
DT_\ve(x) = I_{n\times n}  -  B(\tfrac{x'}{\ve}) + O(\ve^{\frac12}).
$$
The following lemma gives the precise asymptotic behavior of the tangential Jacobian in this case.
\begin{lemp}\label{Jepsiloncriticodebil*m}
Given $g\in L^1(\partial \Omega)$ we have
$$
\int_{\partial\Omega} g J_\tau T_\ve^{-1} \,dS\to \int_{\partial\Omega} g m\,dS,\text{ si }\ve\to 0.
$$
That is $J_\tau T_\ve^{-1}\stackrel{*}{\rightharpoonup} m$ weakly-* in $L^\infty(\partial\Omega)$, where $m$ is the function defined by \eqref{m}.
\end{lemp}

\begin{proof}
Let $g\in C(\partial\Omega)$ be arbitrary. We first analyze the convergence locally, so we recall the construction of the perturbations. Then, let $U\subset \R^n$ be as in \eqref{Omega.epsilon} and assume that $\supp(g)\subset U$. We then have that
\begin{equation}\label{convergenciaJepsiloncritico}
\begin{split}
\int_{\partial\Omega \cap U}g J_\tau T_\ve^{-1}\,dS &=\int_{\partial\Omega_\ve \cap U}(g\circ T_\ve)\,dS\\
&= \int_{U'}(g\circ T_\ve)(x') \sqrt{1+|\nabla \Phi(x') +\nabla f(\tfrac{x'}{\ve})|^2}\,dx'. 
\end{split}
\end{equation}
But now
\begin{align*}
\int_{U'} (g\circ T_\ve)(x') \sqrt{1+|\nabla \Phi(x') +\nabla f(\tfrac{x'}{\ve})|^2}\,dx' = \int_{U'} (g\circ T_{\ve})(x') m_\ve(x')\sqrt{1+|\nabla \Phi(x')|^{2}}\,dx', 
\end{align*}
where 
$$
m_\ve(x') := m(x', \tfrac{x'}{\ve}),\quad
m(x',y) :=\frac{\sqrt{1+|\nabla \Phi(x')+\nabla f(y)|^{2}}}{\sqrt{1+|\nabla \Phi(x')|^{2}}}.
$$
Using that $f$ is periodic with period $Y$, it follows that $m(x',y)$ is periodic in $y$ with period $Y$ and hence
$$
m_\ve \stackrel{*}{\rightharpoonup} m \quad \text{weakly-* in } L^\ito(\R^{n-1}). 
$$
See \cite{Lions}.

On the other hand, since $T_\ve\rightrightarrows id_{\R^n}$ if follows that $(g\circ T_\ve)\rightrightarrows g$ uniformly on compact sets, in particular, $(g\circ T_\ve)\to g$ in $L^1(U')$.

Combining all these facts, we arrive at
$$
\int_{U'} (g\circ T_\ve) m_\ve \sqrt{1+|\nabla\Phi|^2}\, dx'\to \int_{U'} g m \sqrt{1+|\nabla\Phi|^2}\, dx' = \int_{\partial \Omega} g m\, dS.
$$

The case where $g\in C(\partial\Omega)$ is arbitrary, follows by a standard arguments using the partition of unity and is omitted.

Finally, if $g\in L^1(\partial\Omega)$ a standard approximation argument gives the desired result.
\end{proof}

Summing up we have proved the following result for the perturbation \eqref{cambiovariablesupercritico}.

\begin{thm}\label{estimacionTve}
Let $\{T_\ve\}_{\ve>0}$ be the transformation given by \eqref{cambiovariablesupercritico}.Then the following estimates hold:
\begin{enumerate}
\item If $a>1$, $T_\ve \to id_{\R^n}$ in $C^1$ norm as $\ve\to 0$. In consequence
$$
T_\ve\rightrightarrows id_{\R^n},\ DT_\ve \rightrightarrows I_{n\times n},\ JT_\ve\rightrightarrows 1 \text{ and } J_\tau T_\ve \rightrightarrows 1.
$$

\item If $a=1$, we have that for any compact set $K\subset \Omega$ there exists $\ve_0>0$ such that
$$
T_\ve|_K = id_K,
$$
for every $0<\ve<\ve_0$. Moreover, 
$$
J_\tau T_\ve^{-1} \stackrel{*}{\rightharpoonup} m \quad\text{weakly-* in } L^\ito(\partial\Omega),
$$
where $m$ is the function given by \eqref{m}.
\end{enumerate}
\end{thm}

\section{Subcritical case ($a<1$)}\label{subcritical.sec}
In this section we prove the result in the subcritical case. This is the simplest of the three cases.

\begin{proof}
Let us take $\Gamma_0\subset \partial\Omega$ as the closure of a relative open and connected set such that $|\Gamma_0|_{n-1} > \alpha |\partial\Omega|_{n-1}$. 

Given $\delta>0$, consider the sets $U_\delta = B_\delta(\Gamma_0)$ defined as
$$
U_{\delta} := \{x\in\R^n \colon \dist(x,\Gamma_0)<\delta\}
$$
and take $\Gamma_1\subset \partial\Omega\setminus \overline{U_{2\delta}}$ such that $|\Gamma_1|_{N-1}>0$.

Let now $\phi\in C^1(\bar{\Omega})$ be such that $\phi\equiv 0$ in $U_\delta$, $\phi\equiv 1$ in $\Omega\setminus U_{2\delta}$ and $0\le \phi\le 1$, $|\nabla \phi|\le C\delta^{-1}$ in $U_{2\delta}\setminus U_\delta$.

Observe that if we denote by $\Gamma_{0,\ve}\subset \partial\Omega_\ve$ to the portion of the boundary of $\Omega_\ve$ that comes from perturbing $\Gamma_0$, one has that $\phi\equiv 0$ in $\Gamma_{0,\ve}$ for every $\ve>0$ small. Moreover, is easy to see that $|\Gamma_{0,\ve}|_{n-1}\ge \alpha |\partial\Omega_\ve|_{n-1}$. Then, $\phi$ is admissible in the characterization of $\lambda_\ve(\alpha)$. As a consequence, we get the following estimate:
$$
\lambda_\ve(\alpha)\le \frac{\int_{\Omega_\ve} |\nabla \phi|^p + |\phi|^p\, dx}{\int_{\partial\Omega_\ve} |\phi|^p\, dS}.
$$
This quotient can be easily estimated. In fact,
\begin{equation}\label{cota.phi}
\int_{\Omega_\ve}|\nabla \phi|^p \,dx\le C|\Omega_\ve|_n,\quad \int_{\Omega_\ve}|\phi|^p\,dx\le |\Omega_\ve|_n,
\end{equation}
with $C=C(\delta)$.

On the other hand,
\begin{equation}\label{cota.phi.borde}
\int_{\partial \Omega_\ve} |\phi|^p\, dS \ge \int_{\partial \Omega_\ve\setminus \bar U_{2\delta}} |\phi|^p\, dS = |\partial \Omega_\ve\setminus \bar U_{2\delta}|_{n-1}\ge |\Gamma_{1,\ve}|_{n-1},
\end{equation}
where $\Gamma_{1,\ve}$ stands for the perturbed set obtained from $\Gamma_1\subset \partial\Omega\setminus \bar U_{2\delta}$.

But,
\begin{align*}
|\Gamma_{1,\ve}|_{n-1} &= \int_{U'}\sqrt{1+\left |\nabla \Phi(x')+\ve^{a-1}\nabla f(\tfrac{x'}{\ve})\right |^{2}}\,dx'\\
&= \ve^{a-1} \int_{U'}\sqrt{\ve^{2(1-a)}+\left |\ve^{1-a}\nabla \Phi(x')+\nabla f(\tfrac{x'}{\ve})\right |^{2}}\,dx'.
\end{align*}

Let us now estimate this last integral.
\begin{align*}
\int_{U'}&\sqrt{\ve^{2(1-a)}+\left |\ve^{1-a}\nabla \Phi(x')+\nabla f(\tfrac{x'}{\ve})\right |^{2}}\,dx' \\
&= \int_{U'} \left (\sqrt{\ve^{2(1-a)}+\left |\ve^{1-a}\nabla \Phi(x')+\nabla f(\tfrac{x'}{\ve})\right |^{2}}-|\nabla f(\tfrac{x'}{\ve})|\right )+|\nabla f(\tfrac{x'}{\ve})|\,dx'.
\end{align*}
If we now denote by $\rho_{\ve}(x') = \sqrt{\ve^{2(1-a)}+\left |\ve^{1-a}\nabla \Phi(x') + \nabla f(\tfrac{x'}{\ve})\right |^{2}}-|\nabla f(\tfrac{x'}{\ve})|$, it is not difficult to see that
$$
|\rho_{\ve}(x')|\leq \ve^{1-a}(1+|\nabla\Phi(x')|)
$$
from where it follows that
$$
\int_{U'} \left( \sqrt{\ve^{2(1-a)}+\left |\ve^{1-a}\nabla \Phi(x') + \nabla f(\tfrac{x'}{\ve})\right |^{2}} - |\nabla f(\tfrac{x'}{\ve})|\right )dx'\rightarrow 0\quad \text{when } \ve\rightarrow 0.
$$

Finally, by the periodicity of $f$, we conclude that
$$
\int_{U'} |\nabla  f(\tfrac{x'}{\ve})|\, dx' \rightarrow \int_Y |\nabla f(y)|\, dy =: \overline{|\nabla f|}>0.
$$

These estimates allow us to conclude that,
\begin{equation}\label{cota.gamma}
|\Gamma_{1,\ve}|_{N-1}\ge \ve^{a-1}\frac{\overline{|\nabla f|}}{2},
\end{equation}
for every $\ve>0$ small.

Now, from \eqref{cota.phi}, \eqref{cota.phi.borde} and \eqref{cota.gamma}, we obtain
$$
\lambda_\ve(\alpha) \le C \ve^{1-a} \to 0 \quad \text{when } \ve\to 0
$$
as we wanted to show. 
\end{proof}
\section{Supercritical and critical cases ($a\le1$)}\label{critical.sec}
Now taking into account Theorem \ref{estimacionTve}, we note that the supercritical and critical cases in Theorem \ref{teo.main} are special cases of a more general result. 

Indeed if $T_\ve\colon\Omega_\ve \to \Omega$ is a family of perturbations which satisfies the following condition:
\begin{equation}\label{propiedades.perturbaciones}
\begin{cases}
T_\ve = id_{\R^n}. & \quad \text{on each compact set } K\subset \Omega \text{ for }\ve<\ve_0(K)\\
J_\tau T_\ve^{-1} \stackrel{*}{\rightharpoonup}m,& \quad\text{weakly* in } L^\infty(\partial\Omega) \text{ when }\ve\to 0,
\end{cases}
\end{equation}
where $m\in L^\infty(\partial\Omega)$ then we get the following general result.
\begin{thm}\label{superycriticogeneral}
Let $\{T_\ve\}_{\ve>0}$ be a family of perturbations that satisfies condition \eqref{propiedades.perturbaciones}. Then 
$$
\lambda_\ve (\alpha)\to \lambda^*(\alpha),\text{ when }\ve\to  0,
$$
where $\lambda_\ve (\alpha)$ is given by \eqref{lambda.alpha} on $\Omega_\ve$ and $\lambda^*(\alpha) $ is given by
$$
\lambda^* (\alpha)=\inf\left \{\frac{\int_{\Omega}|\nabla u|^p+|u|^p\,dx}{\int_{\partial\Omega}|u|^p\,d\mu^*}\colon u\in W^{1,p}(\Omega),\ \mu^*(\{u=0\}\cap \partial\Omega)\geq \alpha \mu^*(\partial \Omega)\right \}.
$$
Here the measure $\mu^*$ is given by $d\mu^*=m\,dS$.

Moreover, if $u_\ve$ is an eigenfunction associated to $\lambda_\ve(\alpha)$ normalized as $\|u_\ve\|_{L^p(\partial\Omega_\ve)}=1$, then the sequence $\{u_\ve\circ T_\ve^{-1}\}_{\ve>0}\subset W^{1,p}(\Omega)$ is weakly pre compact and every accumulation point is an eigenfunction of $\lambda^*(\alpha)$.
\end{thm}

Clearly, Theorem \ref{superycriticogeneral} implies the critical ($a=1$) and supercritical ($a>1$) cases in Theorem \ref{teo.main}. Also Theorem \ref{superycriticogeneral} implies Theorem 6.2 in \cite{Bonder-Neves-Delpezzo}.

Before starting the proof we need the following observations.

Let $\Omega_1, \Omega_2 \subset \R^n$ be open domains and suppose that there exists a diffeomorphism $T\colon\Omega_1 \to \Omega_2$. This diffeomorphism $T$ induces the mapping 
$$
\T\colon W^{1,p}(\Omega_2)\to W^{1,p}(\Omega_1), \quad \T(u)=u\circ T.
$$ 
This mapping is linear, continuous and invertible, with $\T^{-1}v=v\circ T^{-1}$. Moreover a direct application of the Change of Variables Theorem implies that 
\begin{equation}\label{tauLp}
\int_{\Omega_1} |\T u|^p\, dx\le \|JT^{-1}\|_\infty \int_{\Omega_2} |u|^p\, dx
\end{equation}
and 
\begin{equation}\label{tauWp}
\begin{split}
\int_{\Omega_1}|\nabla(\T u)|^p\, dx &\le \|JT^{-1}\|_\infty \|DT\|_\infty \int_{\Omega_2} |\nabla u|^p\, dy.
\end{split}
\end{equation}
Then if we consider now the general pertubations $T_\ve\colon\Omega_\ve \to \Omega$ which satisfies the properties \eqref{propiedades.perturbaciones} we get the associated mappings $\T_\ve\colon W^{1,p}(\Omega)\to W^{1,p}(\Omega_\ve)$, which are linear, invertible and, by \eqref{tauLp} and \eqref{tauWp}, bi-continuous.

With this in mind we define the functions $Q_\ve \colon W^{1,p}(\Omega_\ve)\to \R$, $Q\colon W^{1,p}(\Omega)\to \R$ by
\begin{equation}\label{funperturbado}
Q_{\ve}(u)=\int_{\Omega_{\ve}}|\nabla{u}|^{p}+|u|^{p}dx
\end{equation}
and
\begin{equation}\label{funsinperturbar}
Q(v)=\int_{\Omega}|\nabla v|^p + |v|^p \, dy.
\end{equation}
We now consider the function $\tilde{Q}_\ve:W^{1,p}(\Omega)\to \R$ defined by $\tilde{Q}_\ve=Q_\ve\circ \T_\ve^{-1}$.

We introduce the sets  
\begin{equation}
\label{Xepalpha} X_\alpha^\ve := \{u\in W^{1,p}(\Omega_\ve)\colon |\{u=0\}\cap \partial\Omega_\ve|_{n-1}\ge \alpha |\partial\Omega_\ve|_{n-1} \text{ and } \|u_\ve\|_{L^p(\partial\Omega_\ve)}=1\},
\end{equation}
\begin{equation}\label{tildeXealpha}
\tilde{X}_\alpha^\ve := \T_\ve(X_\alpha^\ve) = \{v\in W^{1,p}(\Omega)\colon v\circ T_\ve^{-1} \in X_\alpha^\ve\}.
\end{equation}
\begin{equation}\label{X*alpha}
X_\alpha^* := \{v\in W^{1,p}(\Omega)\colon \mu^*(\{v=0\}\cap \partial\Omega)\ge \alpha \mu^*(\partial\Omega) \text{ and } \|v\|_{L^p(d\mu^*)}=1\},
\end{equation}
where $d\mu^*=m\,dS$.

With the above notations, we can write
$$
\lambda_\ve(\alpha) = \inf_{u\in X_\alpha^\ve} Q_\ve(u) = \inf_{v\in \tilde{X}_\alpha^\ve} \tilde{Q}_\ve(v)\quad \text{and}\quad \lambda^*(\alpha) = \inf_{v\in X^*_\alpha} Q(v).
$$

In order to prove the convergence of these minima, we use the notion of {\em $\Gamma-$conver-gence}. This notion was introduced by E. De Giorgi in the 60's and is by now a classical subject in dealing with variational problems. We refer the reader to  the books of A. Braides \cite{Braides} and of G. Dal Maso \cite{DalMaso}.

For the sake of completeness, we recall the definition of $\Gamma-$convergence.
\begin{de}
Let $(X,d)$ be a metric space and let $J_\ve, J\colon X\to (-\infty, +\infty]$. We say that $J_\ve$ $\Gamma-$converges to $J$ as $\ve\to 0$ if
\begin{itemize}
\item ($\liminf$ inequality) For every $x\in X$ and for every sequence $\{x_\ve\}_{\ve>0}\subset X$ such that $x_\ve\to x$, we have
$$
J(x)\le \liminf_{\ve\to 0} J_\ve(x_\ve).
$$

\item ($\limsup$ inequality) For every $x\in X$ there exists $\{y_\ve\}_{\ve}\subset X$ such that $y_\ve\to x$ and
$$
J(x)\ge \limsup_{\ve\to 0} J_\ve(y_\ve).
$$
\end{itemize}
We denote this convergence by $J = \glim_{\ve\to 0} J_\ve$.
\end{de}

This notion is extremely useful in dealing with convergence of minima as the following theorem shows.
\begin{thm}\label{convergencia.minimos}
Let $(X,d)$ be a metric space and let $J_\ve, J\colon X\to (-\infty, +\infty]$ be such that $J = \glim_{\ve\to 0} J_\ve$. Assume that for every $\ve>0$ there exists $x_\ve\in X$ such that $J_\ve(x_\ve)=\inf_X J_\ve$. Moreover, assume that $\{x_\ve\}_{\ve>0}$ is precompact in $X$. Then
\begin{itemize}
\item $\inf_X J = \lim_{\ve\to 0} \inf_X J_\ve$.

\item If $x$ is any accumulation point of the sequence $\{x_\ve\}_{\ve>0}$, then $J(x) = \inf_X J$.
\end{itemize}
\end{thm}
 The proof of this theorem is elemental and can be found in any of the above mentioned books \cite{Braides, DalMaso}.

We apply this theorem to the functions $J_\ve, J\colon L^p(\Omega)\to (-\infty,+\infty]$ given by
\begin{align}\label{Jve}
J_\ve(v) &:= \begin{cases}
\tilde{Q}_\ve(v) & \text{ if } v\in \tilde{X}_\alpha^\ve\\
+\infty & \text{ if not}
\end{cases}\\ \label{J}
J(v) &:= \begin{cases}
Q(v) & \text{ if } v\in X_\alpha^*\\
+\infty & \text{ if not}
\end{cases}
\end{align}

We begin by showing the $\Gamma-$convergence of the functionals. For this we need the following lemmas.

\begin{lemp}\label{lema.medidas1}
Let $(X, \Sigma,\nu)$ be a measurable space of finite measure and let $\{f_k\}_{k\in\N},f$ be $\Sigma-$measurable and nonnegative functions such that $f_k\to f$ $\nu-$a.e. 

Let $\{\mu_k\}_{k\in\N}$ and $\mu$ be absolutely continuous measures with respect to $\nu$ such that $\mu_k(A)\to \mu(A)$, for every $A\in \Sigma$.

Then
$$
\limsup_{k\to \ito}\mu_k(\{f_k =0\})\leq \mu(\{f=0\}).
$$
\end{lemp}

\begin{rk}
When $\mu_k=\mu$ for every $k\in\N$ this is well known with a simple proof. In this case, the difficulty appears since the measures vary. We do not know if this result is known nor if the hypotheses are optimal. Nevertheless it will suffices for our purposes.
\end{rk}

\begin{rk}
By standard arguments, it can be shown that the condition $\mu_n(A)\to \mu(A)$ for every $A\in \Sigma$ is equivalent to the weak convergence of the densities of the measures in $L^1(X,\nu)$.
\end{rk}

\begin{proof}[Proof of Lemma \ref{lema.medidas1}]
Assume by contradiction that there exists $\delta>0$ such that, for all $k_0\in\N$ there exists at least one $k\geq k_0$ such that
$$
\mu(\{f=0\})+\delta < \mu_k(\{f_k=0\}).
$$ 
Since $\{f=0\}=\bigcap_{j=1}^{\ito}\{f\leq 1/j\}$ it follows that $\mu(\{f=0\})=\lim_{j\to \ito}\mu(\{f\leq 1/j\})$. Hence, there exists $j_0\in\N$ such that, for $j\ge j_0$, 
$$
\mu(\{f\leq 1/j\})+\frac{\delta}{2}<\mu_k(\{f_k =0\}).
$$
On the other hand, since $f_k\to f$ $\nu-$a.e., it follows that
$$
\{f\leq 1/j\}\supset \bigcap_{k_0 =1}^{\ito}\bigcup_{k\geq k_0}\{f_k <1/j\},
$$
so
$$
\mu \left( \bigcap_{k_0 =1}^{\ito}\bigcup_{k\geq k_0}\{f_k <1/j\}\right)\leq \mu(\{f\leq 1/j\}.
$$
But, since
$$
\lim_{k_0 \to\ito}\mu\left(\bigcup_{k\geq k_0}\{f_k>1/j\}\right)=\mu \left( \bigcap_{k_0 =1}^{\ito}\bigcup_{k\geq k_0}\{f_k <1/j\}\right),
$$
there exists $k_0(\delta)$ such that
$$
\mu\left ( \bigcup_{k\geq k_0}\{f_k<1/j\} \right )+\frac{\delta}{4}<\mu_{k}(\{f_k =0\}).
$$
Calling $A=\cup_{i\geq k_0}\{f_i<1/j\}$, by hypothesis we have that $\lim_{k\to\ito}\mu_k(A)=\mu(A)$ and therefore, 
$$
\mu_k \left ( \bigcup_{i\geq n_0}\{f_i<1/j\} \right )+\frac{\delta}{8}<\mu_k(\{f_k =0\}).
$$
Finally observe that $\{f_k >1/j\}\subset \bigcup_{i\geq n_0}\{f_i<1/j\}$ and hence we can conclude that
$$
\mu_k(\{f_k <1/j\})+\frac{\delta}{8}< \mu_k(\{f_k =0\}),
$$
a contradiction.
\end{proof}

\begin{lemp}\label{lema.mosco}
Let $\tilde{X}_\alpha^\ve, X_\alpha^*\subset W^{1,p}(\Omega)$ be the sets defined in \eqref{tildeXealpha} and \eqref{X*alpha} respectively. Then, given $v_\ve\in \tilde{X}_\alpha^\ve$ such that $v_\ve\rightharpoonup v$ weakly in $W^{1,p}(\Omega)$, it follows that $v\in X_\alpha^*$. 

Reciprocally, for every $v\in X_\alpha^*$ there exists a sequence $\{\ve_k\}_{k\in\N}$ such that $\ve_k\downarrow 0$ and $v_k\in \tilde{X}_\alpha^{\ve_k}$ such that $v_k\rightharpoonup v$ weakly in $W^{1,p}(\Omega)$. Moreover, the sequence can be taken to converge strongly in $W^{1,p}(\Omega)$.
\end{lemp}

\begin{rk}
The result of the previous Lemma says that the sets $\tilde{X}_\alpha^\ve$ converges in the sense of Mosco to the set $X_\alpha^*$. See \cite{Henrot}.
\end{rk}

\begin{proof}
Let $v\in X_\alpha^*$ and set $\Gamma=\{v=0\}\cap \partial\Omega$. 

Given $k\in \N$ define $\tilde v_k := \max\{v-\tfrac{1}{k},0\}$. Then, $\Gamma_k = \{\tilde v_k=0\}\cap \partial\Omega$ verifies that $\mu^*(\Gamma_k)>\mu^*(\Gamma)$ (recall that the weight $m$ is strictly positive). So, there exists $\rho_k>0$ such that
\begin{equation}\label{mu*1}
\mu^*(\Gamma_k) \ge (1 +\rho_k)\alpha \mu^*(\partial\Omega).
\end{equation}

It is straightforward to check that $\tilde v_k \to v$ strongly in $W^{1,p}(\Omega)$ as $k\to \infty$.

Now let $t_{k, \ve}>0$ be such that $v_{k, \ve} := t_{k, \ve} \tilde v_k$ verifies that $\|v_{k, \ve}\circ T_\ve\|_{L^p(\partial\Omega_\ve)} = 1$. It is easy to see that $\tilde v_k \to v$ ($k \to \infty$) strongly in $L^p(\partial\Omega)$ implies that $\lim_{\ve\downarrow 0}(\lim_{k\to \infty}t_{k,\ve}) \to 1$. So, we have that $v_{k,\ve} \to v$ strongly in $W^{1,p}(\Omega)$ as $k\to \infty$ and $\ve\downarrow 0$.

It remains to check that, given $k\in\N$ there exists $\ve_k$ with $\ve_k\downarrow 0$ such that $v_k := v_{k,\ve_k}\in \tilde X_\alpha^{\ve_k}$ and for this we have only to check that
$$
\mu_{\ve_k}(\Gamma_k)\ge \alpha \mu_{\ve_k}(\partial\Omega)
$$
where $\Gamma_k = \{v_k=0\}\cap \partial\Omega = \{\tilde v_k=0\}\cap\partial\Omega$ and $d\mu_{\ve} = J_\tau T_\ve dS$.

 But, since $\mu_{\ve}(A)\to \mu^*(A)$ for every $dS-$measurable set $A\subset \partial\Omega$, we have that there exists $\ve_k$ such that
\begin{equation}\label{mu*2}
\mu_{\ve_k}(\Gamma_k)\ge (1+\rho_k)^{-\frac12} \mu^*(\Gamma_k) \quad\text{and}\quad  \mu^*(\partial\Omega)\ge (1+\rho_k)^{-\frac12} \mu_{\ve_k}(\partial\Omega).
\end{equation}

Combining \eqref{mu*1} and \eqref{mu*2} we arrive at
\begin{align*}
\mu_{\ve_k}(\Gamma_k) &\ge (1+\rho_k)^{-\frac12}\mu^*(\Gamma_k) \\
&\ge(1+\rho_k)^{\frac12} \alpha\mu^*(\partial\Omega)\\
&\ge \alpha \mu_{\ve_k}(\partial\Omega)
\end{align*}

Now, we need to see that if $v_\ve\in \tilde X_\alpha^\ve$ is such that $v_\ve\rightharpoonup v$, then $v\in X_\alpha^*$. But this is an immediate consequence of Lemma \ref{lema.medidas1}. 

In fact, Lemma \ref{lema.medidas1} is applied to the functions $v_\ve, v\in W^{1,p}(\Omega)\subset L^p(\partial\Omega)$ (recall that we can assume that $v_\ve\to v$ $dS-$a.e.on $\partial\Omega$) and the measures
$$
d\mu_\ve = J_\tau T_\ve^{-1} dS,\quad d\mu^* = mdS,\quad d\nu=dS.
$$
As a consequence, we get
\begin{align*}
\mu(\{v=0\}\cap \partial\Omega)&\ge \limsup_{\ve\to 0} \mu_\ve(\{v_\ve=0\}\cap\partial\Omega) = \limsup_{\ve\to 0}\int_{\{v_\ve=0\}\cap\partial\Omega} J_\tau T_\ve^{-1}\, dS\\
&=\limsup_{\ve\to 0}|\{v_\ve\circ T_\ve=0\}\cap \partial\Omega_\ve|_{N-1}\\
&\ge \limsup_{\ve\to 0} \alpha |\partial\Omega_\ve|_{N-1} = \alpha \mu(\partial\Omega).
\end{align*}
This finishes the proof.
\end{proof}

Unfortunately, we are not able to prove the $\Gamma-$convergence of the functionals in its full generality. In fact we can only prove $\Gamma-$convergence for the supercritical case, that in this general setting will be in the case where $T_\ve \to id_{\R^n}$ in the $C^1$ topology. 

For the more general setting of \eqref{propiedades.perturbaciones}, we can prove a weaker version of $\Gamma-$convergence under which Theorem \ref{convergencia.minimos} still holds. Namely
\begin{prop}\label{g-convergencia}
Let $J_\ve, J\colon L^p(\Omega)\to (-\infty, +\infty]$ be the functionals defined by \eqref{Jve}-\eqref{J}. Assume that the transformations $T_\ve$ verify \eqref{propiedades.perturbaciones}.  Then:
\begin{itemize}
\item for every sequence $\{v_\ve\}_{\ve>0}\subset L^p(\Omega)$ of minimizers of $\{J_\ve\}_{\ve>0}$ such that $v_\ve\to v$ in $L^p(\Omega)$, we have that
$$
J(v)\le \liminf_{\ve\to 0} J_\ve(v_\ve).
$$

\item Moreover, for every $v\in L^p(\Omega)$, there exists $v_k\in L^p(\Omega)$ and $\ve_k\downarrow 0$ such that $v_k\to v$ in $L^p(\Omega)$ and
$$
J(v)\ge \limsup_{k\to\infty} J_{\ve_k}(v_k).
$$
\end{itemize}
\end{prop}

\begin{rk}
Observe that the only difference with respect to $\Gamma-$convergence is that we do not prove the liminf inequality for {\em every} sequence $\{v_\ve\}_{\ve>0}$, but only for sequences of minimizers. It is straightforward to check that the conclusion of Theorem \ref{convergencia.minimos} still holds under this weaker assumption.
\end{rk}

\begin{proof}
We will divide the proof into two parts.

\noindent {\em lim inf inequality:} Let $v_\ve\in \tilde{X}_\alpha^\ve$ be such that $\tilde Q_\ve(v_\ve) = \inf_{\tilde X_\alpha^\ve} \tilde Q_\ve$ and assume that  $v_{\ve}\to v$ in $L^p(\Omega)$ as  $\ve \downarrow 0$ for some $v\in X^*_\alpha$. We can assume that
\begin{equation}\label{cota.supuesta}
\liminf_{\ve \rightarrow 0}\tilde{Q}_{\ve}(v_{\ve})<+\infty,
\end{equation}
otherwise there is nothing to prove. It is immediate to see that \eqref{tauLp} and \eqref{tauWp} imply that 
$$
\|v_\ve\|_{W^{1,p}(\Omega)} = Q(v_\ve) \le C \tilde Q_\ve(v_\ve)
$$
and so, by \eqref{cota.supuesta} we conclude that $\{v_\ve\}_{\ve>0}$ is bounded  on $W^{1,p}(\Omega)$ and since $v_\ve\to v$ in $L^p(\Omega)$ it easily follows that $v_\ve\rightharpoonup v$ weakly in $W^{1,p}(\Omega)$.

Since $v_\ve\to v$ strongly in $L^p(\Omega)$ and $J T_\ve\to 1$ a.e. in $\Omega$ and are uniformly bounded it follows that
\begin{equation}\label{convergenciagmadebilinferior1}
\int_{\Omega}|v_\ve|^p JT_\ve^{-1}\, dx\to \int_{\Omega}|v|^p\,dx,\text{ as }\ve\downarrow 0. 
\end{equation}

Taking into account the expressions of $\tilde{Q}_\ve(v_\ve)$ and $Q(v)$, it remains to show that
\begin{equation}\label{inferiorcritico}
\liminf_{\ve\to 0}\int_\Omega |\nabla v_\ve (DT_\ve\circ T_\ve^{-1})|^p JT_\ve^{-1}\, dx\geq \int_{\Omega}|\nabla v|^{p}\,dx.
\end{equation}

We first show that $\nabla v_\ve \to \nabla v$ a.e. in $\Omega$. To this end we need the fact that the sequence $\{v_\ve\}_{\ve>0}$ is a sequence of minimizers for $\tilde Q_\ve$ and therefore they verify the Euler-Lagrange equation associated to the functional $\tilde Q_\ve$. That is
\begin{align*}
\int_{\Omega}\big(|\nabla v_\ve (DT_\ve\circ T_\ve^{-1})|^{p-2}\nabla v_\ve (DT_\ve\circ T_\ve^{-1})\cdot\nabla \psi + &|v_\ve|^{p-2} v_\ve\psi\big)J T_\ve^{-1}\,dx\\
&= \lambda_\ve(\alpha) \int_{\partial\Omega}|v_\ve|^{p-2}v_\ve\psi J_\tau T_\ve^{-1}\,dS
\end{align*}
for every $\psi\in W_\Gamma^{1,p}(\Omega)$. 

Take now $K\subset \Omega$ a compact set, let $\delta=\frac12 d(K,\partial\Omega)$ and write $K_\delta =\{x\in \Omega\colon d(x,K)<\delta\}$. Therefore $K\subset K_\delta \subset\subset \Omega$ and if $\ve$ is small enough, we have that $T_\ve = id_{\R^n}$ on $K_\delta$.

Observe now that if $\psi_0\in W^{1,p}_\Gamma(\Omega)$ is such that $\supp(\psi_0)\subset K_\delta$, then
$$
\int_{\Omega}|\nabla v_\ve|^{p-2}\nabla v_\ve \nabla \psi_0+|v_\ve|^{p-2}v_\ve\psi_0\,dx=0.
$$

So consider $\eta\in C_c^\ito(\Omega)$ be such that $\eta=1$ in $K$, $\supp(\eta)\subset K_\delta$ and $0\leq \eta\leq 1$ in $K_{\delta}\setminus K$.

Therefore, for $\psi_\ve =\eta (v_\ve -v)$, we have
$$
\int_{\Omega}|\nabla v_\ve|^{p-2}\nabla v_\ve \nabla \psi_\ve+|v_\ve|^{p-2}v_\ve\psi_\ve\,dx=0,
$$
that is
\begin{equation}\label{debilconphiepsilon}
\int_{\Omega}|\nabla v_\ve|^{p-2}\nabla v_\ve (v_\ve-v)\nabla \eta+|\nabla v_\ve|^{p-2}\nabla v_\ve\nabla(v_\ve-v)\eta +|v_\ve|^{p-2}v_\ve\eta (v_\ve-v)dx =0.
\end{equation}
Since $v_\ve \rightharpoonup v$ weakly in $W^{1,p}(\Omega)$ we have that $\|\nabla v_\ve\|_{L^p (\Omega)}\leq C$, so
\begin{equation}\label{convgrad1}
\left |\int_{\Omega}|\nabla v_\ve|^{p-2}\nabla v_\ve (v_\ve-v)\nabla \eta\,dx\right |\leq \|\nabla\eta\|_\ito C \|v_\ve -v\|_{L^p(\Omega)}.
\end{equation}
On the other hand, since $\|v_\ve\|_{L^p (\Omega)}\leq C$,
\begin{equation}\label{convgrad2}
\left |\int_{\Omega}|v_\ve|^{p-2}v_\ve \eta (v_\ve -v)\,dx\right | \leq \|\eta\|_\ito C \|v_\ve -v\|_{L^{p}(\Omega)}.
\end{equation}

From \eqref{debilconphiepsilon}, \eqref{convgrad1} and \eqref{convgrad2} we obtain
\begin{equation}\label{convgrad3}
\lim_{\ve\to 0}\int_{K_\delta} |\nabla v_\ve|^{p-2}\nabla v_\ve\nabla(v_\ve-v)\eta\,dx= 0.
\end{equation}

Moreover, since $v_\ve\rightharpoonup v$ weakly in $W^{1,p}(\Omega)$ we get
\begin{equation}\label{convgrad4}
\lim_{\ve\to 0}\int_{K_\delta} |\nabla v|^{p-2}\nabla v\nabla(v_\ve-v)\eta\,dx = 0.
\end{equation}
Combining \eqref{convgrad3} and \eqref{convgrad4} we arrive at
$$
\lim_{\ve\to 0}\int_{K_\delta} (|\nabla v_\ve|^{p-2}\nabla v_\ve-|\nabla v|^{p-2}\nabla v) \nabla(v_\ve-v)\eta\,dx= 0
$$
But now, it is a well known fact (see e.g. \cite{Simon}) that the integrand is nonnegative and therefore $(|\nabla v_\ve|^{p-2}\nabla v_\ve-|\nabla v|^{p-2}\nabla v) \nabla(v_\ve-v)\to 0$ a.e. in $K$. From this, we can easily conclude that $\nabla v_\ve \to\nabla v$ a.e. in $K$. Since $K$ is arbitrary in $\Omega$ we conclude the pointwise convergence of the gradients a.e. in $\Omega$.

From the pointwise convergence of the gradient the conclusion of the liminf inequality follows easily. In fact, since $JT_\ve \to 1$ and $DT_\ve \to I$ a.e. in $\Omega$ we have
$$
|\nabla v_\ve (DT_\ve\circ T_\ve^{-1})|^p JT_\ve^{-1}\to |\nabla v|^p\text{ a.e. in }\Omega. 
$$ 
This last fact, together with Fatou's Lemma imply \eqref{inferiorcritico}.

\noindent {\em lim sup inequality:} Given $v\in X^*_\alpha$, let $v_k \in \tilde{X}_\alpha^{\ve_k}$ be such that $v_k\to v$ strongly in $W^{1,p}(\Omega)$. Observe that such a sequence exists by Lemma \ref{lema.mosco}.

Now this and our hypotheses on $T_\ve$ easily imply that
$$
\lim_{k\to\infty} \tilde{Q}_{\ve_k}(v_k)= Q(v).
$$
The proof is completed.
\end{proof}

Now, the proof of Theorem \ref{superycriticogeneral} follows as a simple corollary.

\begin{proof}[Proof of Theorem \ref{superycriticogeneral}]
The proof is now a trivial consequence of Proposition \ref{g-convergencia} and Theorem \ref{convergencia.minimos}.
\end{proof}

\section{Convergence of optimal windows}\label{window.sec}

In this section we analyze the behavior of a sequence of optimal windows $\{\Gamma_\ve\}_{\ve>0}$. Recall that an optimal windows is a set $\Gamma_\ve\subset \Omega_\ve$ such that $|\Gamma_\ve|_{n-1} = \alpha |\partial\Omega_\ve|_{n-1}$ and $\lambda_\ve(\Gamma_\ve) = \lambda_\ve(\alpha)$. 

We will see that, as a consequence of the convergence of the constants $\lambda_\ve(\alpha)\to \lambda^*(\alpha)$ we will deduce the convergence of these optimal windows to an optimal window of the limit problem in a suitable sense.

\begin{thm}\label{teo.ventanas}
Under the same assumptions and notations of Theorem \ref{superycriticogeneral}, if $\Gamma_\ve\subset \Omega_\ve$ is an optimal window associated to $\lambda_\ve(\alpha)$ then, up to a subsequence, it converges, as $\ve\downarrow 0$, to an optimal window of the limit problem $\lambda^*(\alpha)$ in the following sense: Let us define the Radon measures $\{\nu_\ve\}_{\ve>0}$ as
$$
d\nu_\ve = \chi_{_{\Gamma_\ve}}dS.
$$
Then, the family is pre compact in the weak topology of measures and every accumulation point of $\{\nu_\ve\}_{\ve>0}$ is of the form
$$
d\nu^* = \chi_{_{\Gamma^*}}m dS,
$$
where $\Gamma^*$ is an optimal windows for the problem \eqref{lambdaestrella}.
\end{thm}

In order to show the convergence of optimal windows we need a couple of lemmas.

\begin{lemp}\label{convergenciamedidascritico}
Let $(X, \Sigma,\nu)$ be a measure space of finite measure and let $\{f_n\}_{n\in\N},f$ be $\nu-$measurable nonnegative functions such that $f_n\to f$ $\nu-$a.e. 

Let $\{\mu_n\}_{n\in\N}$ and $\mu$ be nonnegative measures, absolutely continuous with respect to $\nu$ such that $\mu_n(A)\to \mu(A)$, for every $A\in \Sigma$.

Then, if $\lim_{n\to\ito}\mu_n(\{f_n=0\}) = \mu(\{f=0\})$, given $\ve>0$ there exists $j_0\in\N$ such that, for every $j\ge j_0$,
$$
\limsup_{n\to\ito} \mu_n(\{0<f_n\le\frac{1}{j}\}) < \ve.
$$
\end{lemp}

\begin{proof}
Let $\ve>0$. Since $\chi_{_{\{0<f\le\frac{1}{j}\}}}\to 0$ $\nu-$a.e. when $j\to\ito$, we have that there exists $j_0\in\N$ such that
\begin{equation}\label{menorqueep}
\mu(\{0<f\le\frac{1}{j}\})<\ve, \quad\text{for every } j\ge j_0.
\end{equation}

Since $f_n\to f$ $\nu-$a.e., we have that
$$
\{f\le \frac{1}{j}\}\supset \bigcap_{n_0\in\N}\bigcup_{k\ge n_0}\{f_k\le \frac{1}{j}\},
$$
from where
$$
\mu(\{f\le \frac{1}{j}\})\ge \lim_{n_0\to\ito}\mu\left(\bigcup_{k\ge n_0}\{f_k\le \frac{1}{j}\}\right).
$$
Hence, given $\delta>0$, there exists $n_0\in\N$ such that
\begin{equation}\label{mudelta}
\mu(\{f\le \frac{1}{j}\})+\delta \ge \mu\left(\bigcup_{k\ge n_0}\{f_k\le \frac{1}{j}\}\right)
\end{equation}

By our hypothesis on the convergence of the measures,
$$
\lim_{n\to\ito}\mu_n\left(\bigcup_{k\ge n_0}\{f_k\le \frac{1}{j}\}\right) = \mu\left(\bigcup_{k\ge n_0}\{f_k\le \frac{1}{j}\}\right),
$$
from where
\begin{equation}\label{mudelta2}
\mu\left(\bigcup_{k\ge n_0}\{f_k\le \frac{1}{j}\}\right)+\delta\ge \mu_n\left(\bigcup_{k\ge n_0}\{f_k\le \frac{1}{j}\}\right)\ge \mu_n(\{f_n\le \frac{1}{j}\}),
\end{equation}
for any $n$ large enough.

Using \eqref{mudelta} and \eqref{mudelta2} we obtain
$$
\mu(\{f\le \frac{1}{j}\}) + 2\delta \ge \limsup_{n\to\ito}\mu_n(\{f_n\le \frac{1}{j}\}) 
$$
and since $\delta>0$ is arbitrary, it follows that
\begin{equation}\label{casiestoy}
 \limsup_{n\to\ito}\mu_n(\{f_n\le \frac{1}{j}\}) \le\mu(\{f\le \frac{1}{j}\}) 
\end{equation}

Now, the lemma follows from \eqref{menorqueep} and \eqref{casiestoy} by using the hypothesis
$$
\lim_{n\to\ito}\mu_n(\{f_n=0\}) = \mu(\{f=0\}).
$$
The proof is completed.
\end{proof}

\begin{lemp}\label{convergenciadifsimetricacritico}
Let $(X, \Sigma,\nu)$ be a measure space of finite measure and let $\{f_n\}_{n\in\N},f$ be $\nu-$measurable nonnegative functions such that $f_n\to f$ $\nu-$a.e. 

Let $\{\mu_n\}_{n\in\N}$ and $\mu$ be nonnegative measures, absolutely continuous with respect to $\nu$ such that $\mu_n(A)\to \mu(A)$, for every $A\in \Sigma$.

Then, if $\lim_{n\to\ito}\mu_n(\{f_n=0\}) = \mu(\{f=0\})$, it follows that
$$
\lim_{n\to\ito}\mu_n(\{f_n =0\}\Delta \{f=0\})=0.
$$
\end{lemp}

\begin{proof}
By Egoroff's Theorem, we have that, given $\delta >0$, there exists a measurable set $C_\delta \subset X$ such that
$$
f_n \rightrightarrows f,\text{ uniformly when } n \to \ito \text{ in } X \setminus C_\delta  
$$
with
$$
\mu(C_\delta) <\delta. 
$$
Observe that, as $\mu_n(A)\to \mu(A)$ for every $A$ measurable, we can assume that 
$$
\mu_n(C_\delta)<\delta
$$
for every $n$ large enough.
 
Define now the set $E_\delta = X\setminus C_\delta$ and using this uniform convergence on the set $E_\delta$, we have
$$
\{f_n =0\}\cap E_\delta \subset \{f\leq \delta\} \cap E_\delta,
$$
for $\ve$ small enough.

We then have that
$$
\{f=0\}\setminus \{f_n =0\}\subset \left ((\{f\leq \delta \}\setminus \{f_n =0\})\cap E_\delta \right )\cup C_\delta 
$$
from where
$$
\mu_n(\{f=0\}\setminus \{f_n =0\}) \le \mu_n(\{f\le\delta\}) - \mu_n(\{f_n=0\}) + \delta.
$$
Taking the limit as $n\to \ito$, we obtain
$$
\limsup_{n\to \ito}\mu_n(\{f=0\}\setminus \{f_n =0\})\le \mu(\{f\le\delta\}) - \mu(\{f=0\}) + \delta,
$$
and now making $\delta\to 0$ we can conclude
$$
\lim_{n\to \ito} \mu_n(\{f=0\}\setminus \{f_n =0\})=0.
$$

On the other hand, given $j\in\N$, there exists $n_j\in\N$ such that
$$
\{f=0\}\cap E_{j} \subset \{f_{n_j}<\frac{1}{j}\}\cap E_{j},
$$
where $\nu(X\setminus E_j)\le \frac{1}{j}$.

Now, reasoning as in the previous case,
$$
\limsup_{j\to \ito}\mu_{n_j}(\{f_{n_j}=0\}\setminus\{f=0\}) \le \limsup_{j\to \ito} \left( \mu_{n_j}(\{f_{n_j}<\frac{1}{j}\})\right) - \mu(\{f=0\}).
$$
But, from Lemma \ref{convergenciamedidascritico}, it follows that
$$
\mu_{n_j}(\{f_{n_j}<\frac{1}{j}\}) = \mu_{n_j}(\{f_{n_j}=0\}) + \mu_{n_j}(\{0<f_{n_j}<\frac{1}{j}\})\to \mu(\{f=0\}).
$$
This completes the proof.
\end{proof}

\begin{rk}
When the sequence of measures $\mu_n$ is constant, this Lemma was proved in \cite[Lemma 3.1]{Bonder-Groisman-Rossi}.
\end{rk}

With the help of Lemma \ref{convergenciadifsimetricacritico} we can now prove Theorem \ref{teo.ventanas}

\begin{proof}[Proof of Theorem \ref{teo.ventanas}]
Let $u_\ve\in W^{1,p}(\Omega_\ve)$ be an extremal for $\lambda_\ve(\alpha)$. We can assume that $u_\ve\in X^\ve_\alpha$. Then, by \cite[Theorem 3.6]{Bonder-Neves-Delpezzo}, we have that $\{u_\ve=0\}\cap \Omega_\ve = \Gamma_\ve$ is an optimal window for $\lambda_\ve(\alpha)$ and hence it verifies $|\Gamma_\ve|_{n-1}=\alpha |\Omega_\ve|_{n-1}$.

Consider now the rescaled functions $v_\ve := u_\ve\circ T_\ve^{-1}$. Then $v_\ve$ is an extremal of $\tilde Q_\ve$ in the set $\tilde X_\alpha^\ve$.

By Theorem \ref{superycriticogeneral}, we can assume that there exists $v\in W^{1,p}(\Omega)$ such that $v_\ve\rightharpoonup v$ weakly in $W^{1,p}(\Omega)$, $v\in X^*_\alpha$ and $v$ is an extremal for  $\lambda^*(\alpha)$. In particular
$$
\mu^*(\{v=0\}) = \alpha \mu^*(\partial\Omega).
$$

On the other hand,
$$
|\{u_\ve=0\}\cap \partial\Omega_\ve|_{n-1} = \int_{\partial\Omega_\ve} \chi_{_{\{u_\ve=0\}}}\, dS = \int_{\partial\Omega} \chi_{_{\{v_\ve=0\}}} J_\tau T_\ve^{-1}\, dS.
$$
So, if we denote by $\mu_\ve$ to the measure $d\mu_\ve = J_\tau T_\ve^{-1}\, dS$ on $\partial\Omega$, we have that
$$
\mu_\ve(\{v_\ve=0\}\cap \partial\Omega) = \alpha |\partial\Omega_\ve|_{n-1} = \alpha \mu_\ve(\partial\Omega),
$$
and since $J_\tau T_\ve^{-1} \stackrel{*}{\rightharpoonup} m$ weakly-* in $L^\ito(\partial\Omega)$, it holds that
$$
\mu_\ve(A)\to \mu^*(A),
$$
for every $A\subset \partial\Omega$ measurable. In particular, $\mu_\ve(\partial\Omega)\to \mu^*(\partial\Omega)$.

All of this discussion leads us to conclude that
$$
\mu_\ve(\{v_\ve=0\}\cap \partial\Omega)\to \mu^*(\{v=0\}\cap \partial\Omega).
$$
Now we are in a position of applying Lemma \ref{convergenciadifsimetricacritico} and conclude that
$$
\mu_\ve([\{v_\ve=0\}\triangle \{v=0\}]\cap \partial\Omega)\to 0.
$$

Now, let $\Gamma_\ve$ be an optimal window and let $u_\ve\in X^\ve_\alpha$ an associated extremal. Let $v_\ve=u_\ve\circ T_\ve^{-1}$ the rescaled extremal as was previously described. Again, we can assume that $v_\ve\to v$ a.e. in $\partial\Omega$ where $v\in X^*_\alpha$ is an extremal associated to $\lambda^*(\alpha)$.

Let $f\in C_b (\R^n)$, then
\begin{align*}
\int fd\nu_\ve - \int fd\nu^* =&\int_{\partial \Omega_\ve} f \chi_{_{\{u_\ve =0\}}}\, dS-\int_{\partial\Omega} f\chi_{_{\{v=0\}}} m\, dS\\
=&\int_{\partial \Omega}(f\circ T_\ve^{-1}) \chi_{_{\{v_\ve=0\}}} \, d\mu_\ve - \int_{\partial\Omega} f\chi_{_{\{v=0\}}} \, d\mu^*\\
=&\int_{\partial\Omega}(\chi_{_{\{v_\ve =0\}}}-\chi_{_{\{v=0\}}})(f\circ T_\ve^{-1}) \,d\mu_\ve\\
& + \int_{\partial \Omega}\chi_{_{\{v=0\}}}[(f\circ T_\ve^{-1})-f]\, d\mu_\ve \\
&+ \int_{\partial\Omega} \chi_{_{\{v=0\}}} f (d\mu_\ve - d\mu^*)\\
=& A_\ve + B_\ve + C_\ve.
\end{align*}
Each of these terms can be easily shown to converge to zero. In fact
$$
|A_\ve|\le \|f\|_\infty \mu_\ve([\{v_\ve=0\}\triangle\{v=0\}]\cap \partial\Omega)\to 0,
$$
by Lemma \ref{convergenciadifsimetricacritico}. On the other hand, 
$$
|B_\ve|\le \|(f\circ T^{-1}_\ve) - f\|_{L^\ito(\partial\Omega)} \mu_\ve(\partial\Omega)\to 0,
$$
since $\mu_\ve(\partial\Omega)$ is convergent (hence bounded) and $f\circ T^{-1}_\ve\rightrightarrows f$ on compact sets.

Finally, using that $\mu_\ve \rightharpoonup \mu^*$ weakly in the sense of measures it follow that
$$
|C_\ve|\to 0.
$$
This completes the proof of the theorem.
\end{proof}

\section*{Acknowledgements}

This paper was partially supported by Universidad de Buenos Aires under grant UBACyT 20020130100283BA, by CONICET under grant PIP 2009 845/10 and by ANPCyT under grant PICT 2012-0153. J. Fern\'andez Bonder is a member of CONICET and J.F. Spedaletti is a doctoral fellow of CONICET.

\bibliographystyle{amsplain}
\bibliography{biblio}
\end{document}